\documentclass[twoside,leqno,10pt]{amsart}
\usepackage{amsfonts}
\usepackage{amsmath}
\usepackage{amscd}
\usepackage{amssymb}
\usepackage{amsthm}
\usepackage{amsrefs}
\usepackage{latexsym}
\usepackage{bbm}
\usepackage{enumerate}
\begin{document}

\newtheorem{theorem}[subsection]{Theorem}
\newtheorem{proposition}[subsection]{Proposition}
\newtheorem{lemma}[subsection]{Lemma}
\newtheorem{corollary}[subsection]{Corollary}
\newtheorem{conjecture}[subsection]{Conjecture}
\newtheorem{prop}[subsection]{Proposition}
\numberwithin{equation}{section}
\newcommand{\mr}{\ensuremath{\mathbb R}}
\newcommand{\dif}{\mathrm{d}}
\newcommand{\intz}{\mathbb{Z}}
\newcommand{\ratq}{\mathbb{Q}}
\newcommand{\natn}{\mathbb{N}}
\newcommand{\comc}{\mathbb{C}}
\newcommand{\rear}{\mathbb{R}}
\newcommand{\prip}{\mathbb{P}}
\newcommand{\uph}{\mathbb{H}}
\newcommand{\fief}{\mathbb{F}}
\newcommand{\majorarc}{\mathfrak{M}}
\newcommand{\minorarc}{\mathfrak{m}}
\newcommand{\sings}{\mathfrak{S}}
\newcommand{\fA}{\ensuremath{\mathfrak A}}
\newcommand{\mn}{\ensuremath{\mathbb N}}
\newcommand{\mq}{\ensuremath{\mathbb Q}}
\newcommand{\half}{\tfrac{1}{2}}
\newcommand{\f}{f\times \chi}
\newcommand{\summ}{\mathop{{\sum}^{\star}}}
\newcommand{\chiq}{\chi \bmod q}
\newcommand{\chidb}{\chi \bmod db}
\newcommand{\chid}{\chi \bmod d}
\newcommand{\sym}{\text{sym}^2}
\newcommand{\hhalf}{\tfrac{1}{2}}
\newcommand{\sumstar}{\sideset{}{^*}\sum}
\newcommand{\sumprime}{\sideset{}{'}\sum}
\newcommand{\sumprimeprime}{\sideset{}{''}\sum}
\newcommand{\V}{V\left(\frac{nm}{q^2}\right)}
\newcommand{\sumi}{\mathop{{\sum}^{\dagger}}}
\newcommand{\mz}{\ensuremath{\mathbb Z}}
\newcommand{\leg}[2]{\left(\frac{#1}{#2}\right)}
\newcommand{\muK}{\mu_{\omega}}

\title{On Hecke eigenvalues at primes of the form $[g(n)]$}
\author{Stephan Baier and Liangyi Zhao}
\maketitle

\begin{abstract}
In this paper, we study the average of the Fourier coefficients of a holomorphic cusp form for the full modular group at primes of the form $[g(n)]$.
\end{abstract}



\section{Introduction}
The estimation of mean-values of arithmetic functions over sparse sequences and the detection of primes in arithmetically 
interesting and sparse sets of natural numbers are often very hard and of great interest to analytic number theorists.
In \cite{BaZh}, we investigated a problem that addresses both of these questions, namely the distribution of Fourier coefficients 
of cusp forms for the full modular group at Piatetski-Shapiro primes. These are primes of the form $\left[n^c\right]$, where 
$c>1$ is fixed. We successfully handled the $c$'s in the range $1<c<8/7$. In this paper, we extend our result in \cite{BaZh} to primes 
of the form $\left[g(n)\right]$, where $g(x)$ is
a general ``nice'' function that grows much faster than a linear function. However, our result will be weaker in the sense that it covers
the result in \cite{BaZh} only for the range $1<c<30/29$. \newline
 
We first introduce some notations and conditions. By $F$ we denote a holomorphic cusp form of weight $\kappa$ for the full modular 
group $SL_2(\intz)$ and by $\lambda_F(n)$ the normalized $n$-th Fourier coefficient of $F$, {\it i.e.} we assume that
$$
F(z)=\sum\limits_{n=1}^{\infty} \lambda_F(n)n^{(\kappa-1)/2}e(nz)
$$
for $\Im z>0$.  We note that the Ramanujan-Petersson conjecture, proved by P. Deligne \cite{Deli, Deli2}, gives a bound for the modulus of 
$\lambda_F$. It states that for any fixed $\varepsilon>0$,
\begin{equation} \label{Ramanujanpetersson}
 \lambda_F(n) \ll d(n) \ll n^{\varepsilon},
\end{equation}
where $d(n)$ is the number of divisors of $n$.  If we assume, in addition, that $F$ is an eigenform of all the Hecke operators, then 
$F$ can be normalized such that $\lambda_F(1)=1$ and with this normalization the implied constant in the first ``$\ll$'' in \eqref{Ramanujanpetersson} 
can be taken to be 1. \newline

Further, we assume that $g: [1,\infty)\rightarrow [1,\infty)$ is a function satisfying the following conditions.
\begin{enumerate}[(i)]
\item $g$ is monotonically increasing.
\item $g$ is infinitely differentiable.
\item $g$ satisfies the inequalities
\begin{equation} \label{gcond}
x\le g(x) \le x^{30/29-\varepsilon}.
\end{equation}
\end{enumerate}
We note that then the inverse function $f:\mbox{range}(g)\rightarrow [1,\infty)$ of $g$ exists and has the following corresponding properties.
\begin{enumerate}[(a)]
\item $f$ is monotonically increasing.
\item $f$ is infinitely differentiable.
\item $f$ satisfies the inequalities
\begin{equation} \label{fcond}
x^{29/30+\varepsilon}\le f(x)\le x. 
\end{equation}
\end{enumerate}
Moreover, we shall also suppose that the derivatives of $f$ satisfy the following conditions.
\begin{enumerate}
\item[(d)] The $k$-th derivative of $f$ satisfies
\begin{equation} \label{fder}
f^{(k)}(x) \asymp \frac{f(x)}{x^k} \quad \mbox{for all $x$ in the image of $g$ and } k\in\mathbb{N},
\end{equation}
where the implied constants depend on $k$ alone.
\item[(e)] The second and third derivatives of $f$ satisfy
\begin{equation} \label{f2ndder}
 2 f''(x) + x f'''(x) \asymp \frac{f(x)}{x^2} \quad \mbox{for all $x$ in the image of $g$.}
 \end{equation}
\end{enumerate}

Furthermore, we denote the set of primes by $\mathbb{P}$. The main result of this paper is the following. 

\begin{theorem} \label{mainresult}
Let $g: [1,\infty)\rightarrow [1,\infty)$ be a function satisfying the conditions (i) -- (iii) above.  Suppoer that the inverse function of $g$ satisfies the condtions (a)--(e) above.  Let $\lambda_F(n)$ 
be the normalized $n$-th Fourier coefficient of a holomorphic cusp form $F$ for the full modular group $SL_2(\intz)$. 
Then there exists a constant $C>0$ depending on $g$ and $F$ such that
\begin{equation} \label{themainresult}
\sum\limits_{\substack{n\le N\\ [g(n)]\in \mathbb{P}}} \lambda_F\left(\left[ g(n) \right]\right) \ll N\exp \left( -C \sqrt{\log N} \right),
\end{equation}
where the implied $\ll$-constant depends on $g$ and $F$.
\end{theorem}

For comparison, our main result in \cite{BaZh} was as follows. 

\begin{theorem} \label{oldmainresult}
Let $1<c<8/7$ and $\lambda_F(n)$ be the normalized $n$-th Fourier coefficient of a holomorphic cusp form $F$ for the full modular 
group  $SL_2(\intz)$.  Then there exists a constant $C>0$ depending on $F$ such that
\begin{equation} \label{themainresult}
\sum\limits_{\substack{n\le N\\ [n^c]\in \mathbbm{P}}} \lambda_F\left(\left[n^c\right]\right) \ll N\exp \left( -C \sqrt{\log N} \right),
\end{equation}
where the implied $\ll$-constant depends on $c$ and the cusp form $F$.
\end{theorem}

Some parts of \cite{BaZh} generalize directly in the present paper, while other parts cannot be carried 
over. We indicate the differences in the following description of our method for the proof of Theorem \ref{mainresult}. 
First, since every cusp form can be written as a linear combination of finitely many Hecke eigenforms, 
it will suffice to prove Theorem \ref{mainresult} for (normalized) Hecke eigenvalues. 
The advantages of working with Hecke eigenvalues are that they are multiplicative and real. Now we make a similar standard reduction 
of the problem to exponential sums with Hecke eigenvalues and the von Mangoldt function as in \cite{BaZh}. Then, just as in 
\cite{BaZh}, we
decompose the von Mangoldt function using a Vaughan-type identity, which leads to type I and type II sums.  
The type II sums are then treated by simply using van der Corput's method for exponential sums. In contrast, in \cite{BaZh}, we used sophisticated estimates
for exponential sums with monomials, which are not applicable in the present, more general situation. 
For the type I sums, we need to estimate smooth exponential sums with Hecke eigenvalues. Since we work with general functions 
$g(x)$ in place of $x^c$, it is not
possible to apply Jutila's method utilized in \cite{BaZh}. Instead, we estimate the said exponential sums using
a Weyl shift and a bound for shifted convolutions of Hecke eigenvalues with a weakly oscillating weight, a result analogous to that of 
W. Duke, J. B. Friedlander and 
H. Iwaniec in \cite{DFI} for the divisor function.  \newline

\noindent{\bf Notations.} The following notations and conventions are used throughout the paper.\newline
$e(z) = \exp (2 \pi i z) = e^{2 \pi i z}$. \newline
$\eta$ and $\varepsilon$ are small positive real numbers, where $\varepsilon$ may not be the same number in each occurance. \newline
$c>1$ is a fixed number and we set $\gamma=1/c$. \newline
$\lambda(n)$ denotes the normalized $n$-th Fourier coefficients of a Hecke eigenform for the full modular group.  In the sequel, we shall suppress the subscript $F$, used in the introduction, since the cusp form is fixed throughout the paper. \newline
$\Lambda(n)$ is the van Mangoldt function. \newline
$d(n)$ is the divisor function.\newline
$k\sim K$ means $K_1\le k \le K_2$ with $K/2\le K_1\le K_2\le 2K$. \newline
$f = O(g)$ or $f \ll g$ means $|f| \leq cg$ for some unspecified positive constant $c$. \newline
$f \asymp g$ means $f \ll g$ and $g \ll f$.\newline
$[x]$ denotes the largest integer not exceeding $x$, and $\psi(x)=x-[x]-1/2$ denotes the saw-tooth function.

\section{Preliminary lemmas}
For the estimation of exponential sums with Hecke eigenvalues, we need the following bound for shifted convolutions of Hecke eigenvalues. 
 
\begin{lemma} \label{shiftedconvolution} Set 
$$
D_g(a,b;h):=\sum\limits_{am\mp bn=h} \lambda(m)\lambda(n)g(am,bn),
$$
where $a,b\ge 1$, $(a,b)=1$, $h\not=0$ and $g$ is a smooth function on $\mathbb{R}^+\times \mathbb{R}^+$ satisfying
$$
x^iy^j g^{(ij)}(x,y) \ll \left(1+\frac{x}{X}\right)^{-1}\left(1+\frac{y}{Y}\right)^{-1}P^{i+j}
$$
with some $P,X,Y\ge 1$ for all $i,j\ge 0$, the implied constant depending on $i,j$ alone. Then
$$
D_g(a,b;h)\ll P^{5/4}\left(X+Y\right)^{1/4}(XY)^{1/4+\varepsilon},
$$ 
where the implied constant depends on $\varepsilon$ only.
\end{lemma}

\begin{proof} In \cite{DFI}, a result analogous to this one was proved for the divisor function $d(n)$ in place of $\lambda(n)$. The same arguments based on the delta-method and the Voronoi summation formula lead to the above result. 
\end{proof}

To reduce our problem to the estimation of exponential sums, we shall use the following approximation of the saw-tooth function $\psi(x)$ due to J. D. Vaaler.

\begin{lemma} [Vaaler] \label{Valer}
For $0<|t|<1$, let
$$
W(t)=\pi t(1-|t|)\cot \pi t +|t|.
$$
Fix a positive integer $J$. For $x\in \mathbb{R}$ define
$$
\psi^*(x):=-\sum\limits_{1\le |j|\le J} (2\pi ij)^{-1} W\left(\frac{j}{J+1}\right)e(jx)
$$
and
$$
\delta(x):=\frac{1}{2J+2}\sum\limits_{|j|\le J} \left(1-\frac{|j|}{J+1}\right)e(jx).
$$
Then $\delta$ is non-negative, and we have
$$
|\psi^*(x)-\psi(x)|\le \delta(x)
$$
for all real numbers $x$.
\end{lemma}

\begin{proof} This is Theorem A6 in \cite{GK} and has its origin in \cite{vaal}. \end{proof}

At several places of the paper, we shall use the following classical estimate for exponential sums due to van der Corput.

\begin{lemma} [van der Corput] \label{vandercorput} Suppose that $f$ is a real valued function with two continuous derivatives on $[N,N_1]$. 
Suppose also that there is some $\lambda>0$ and some $\alpha\ge 1$ such that
$$
\lambda \le |f''(x)| \le \alpha\lambda
$$ 
on $[N,N_1]$, where $N_1\ge N+1$. Then
$$
\sum\limits_{N<n\le N_1} e(f(n)) \ll \alpha (N_1-N)\lambda^{1/2}+\lambda^{-1/2}.
$$
\end{lemma}

\begin{proof} This is Theorem 2.2. in \cite{GK}. 
\end{proof}

The following is the prime number theorem for Hecke eigenvalues which is used to bound the main term.

\begin{lemma} \label{cuspatprimes}
There exists a positive constant $C$, such that
\begin{equation*}
\sum\limits_{n\le N} \Lambda(n)\lambda(n) \ll N \exp \left( - C \sqrt{\log N} \right),
\end{equation*}
where the implied $\ll$-constant and the constant $C$ depend on the cusp form.
\end{lemma}

\begin{proof} This is a special case of the more general Theorem 5.12 in \cite{HIEK}. \end{proof}

To bound the error term, we shall see that it suffices to prove that
\begin{equation} \label{goal0}
\sum\limits_{n\sim N} \Lambda(n)r(n) =O\left(N^{1-\eta}\right)
\end{equation}
for a some fixed $\eta>0$, where $r$ is a certain function involving $\lambda(n)$ and 
an exponential sum.  
The following lemma reduces the above sum containing the von Mangoldt function to so-called type I and type II sums.

\begin{lemma} [Heath-Brown] \label{Heath}
Let $r(n)$ be a complex-valued function defined on the natural numbers.  Suppose that $u$, $v$ and $z$ are real parameters satisfying the conditions
\[ 3\le u<v<z<2N,\ z-1/2\in \mathbb{N},\ z\ge 4u^2,\ N\ge 32z^2u,\ v^3\ge 64N. \]
Suppose further that $1\le Y\le N$ and $XY=N$. Assume that $a_m$ and $b_n$ are complex numbers.  We write
\begin{equation} \label{Kdef}
 K:=\mathop{\sum_{m\sim X} \sum_{ n\sim Y}}_{mn\sim N} a_m r(mn)
\end{equation}
and
\begin{equation} \label{Ldef}
L:=\mathop{\sum_{m\sim X} \sum_{n\sim Y}}_{mn\sim N} a_m b_n r(mn).
\end{equation}
Then the estimate \eqref{goal0} holds if we uniformly have
\[ K\ll N^{1-2\eta}\ \  \mbox{ for } Y\ge z \; \mbox{and any complex} \; a_m\ll 1 \]
and
\[ L\ll N^{1-2\eta}\ \  \mbox{ for } u\le Y\le v \; \mbox{and any complex} \; a_m,b_n\ll 1. \]
\end{lemma}

\begin{proof} This is a consequence of Lemma 3 in \cite{Hea2}. \end{proof}

To separate the variables $m$ an $n$ appearing in the previous Lemma \ref{Heath}, we shall use the following lemmas. The first of them is the multiplicative property of Hecke eigenvalues, and the second of them is a variant of Perron's formula.

\begin{lemma} \label{multhecke}
Hecke eigenvalues are multiplicative and they satisfy the
following relation.
\begin{equation*}
\lambda (mn) = \sum_{d| \gcd(m,n)} \mu (d) \lambda \left( \frac{m}{d} \right) \lambda \left( \frac{n}{d} \right).
\end{equation*}
\end{lemma}

\begin{proof} This Lemma follows by applying the M\"{o}bius inversion formula to the product formula for the Hecke eigenvalues.  See, for example, Proposition 14.9 of \cite{HIEK}. \end{proof}

\begin{lemma} \label{Perron}
Let $0<M\le N<\nu N<\kappa M$ and let $a_m$ be complex numbers with $|a_m|\le 1$. We then have
\begin{equation} \label{perroneq}
\sum\limits_{N<n<\nu N} a_n = \frac{1}{2\pi} \int\limits_{-M}^M \left(\sum\limits_{M<m<\kappa M} a_m m^{-it}\right) N^{it}(\nu^{it}-1)t^{-1} dt \ + \ O(\log(2+M)),
\end{equation}
where the implied $O$-constant depends only on $\kappa$.
\end{lemma}

\begin{proof} This is Lemma 6 in \cite{EFHI}. \end{proof}

To bound a certain error term, we shall need the following.

\begin{lemma} \label{deltaest}
Assume that $1\le N<N+1\le N_1\le 2N$. Define the function $\delta$ as in Lemma \ref{Valer}. If $f$ satisfies 
$$
f(x)\asymp f(N), \quad f'(x)\asymp \frac{f(N)}{N}, \quad f''(x) \asymp \frac{f(N)}{N^2} \quad \mbox{for } N<x\le N_1,
$$
then
$$
\sum\limits_{N<n\le N_1} \delta\left(-f(n)\right)\ll J^{-1}N+J^{1/2}f(N)^{1/2}+J^{-1/2}Nf(N)^{-1/2}.
$$
\end{lemma}

\begin{proof} We prove this along the lines of Lemma \ref{deltaest} on page 48 in \cite{GK}.  Clearly, we have
\begin{equation}
\sum\limits_{N<n\le N_1} \delta\left(-f(n)\right)\ll \frac{1}{J} \sum\limits_{|j|\le J} \left| \sum\limits_{N<n\le N_1}
e(jf(n))\right| \ll \frac{N}{J}+ \frac{1}{J} \cdot \sum\limits_{1\le j\le J} \left| \sum\limits_{N<n\le N_1}
e(jf(n))\right|.
\end{equation}
Using Lemma \ref{vandercorput}, we get, for $j\ge 1$, that
$$
\sum\limits_{N<n\le N_1} e(jf(n)) \ll j^{1/2}f(N)^{1/2}+j^{-1/2}Nf(N)^{-1/2}.
$$  
Putting everything together, it follows that 
$$
\sum\limits_{N<n\le N_1} \delta\left(-f(n)\right)\ll J^{-1}N+J^{1/2}f(N)^{1/2}+J^{-1/2}Nf(N)^{-1/2}.
$$
Thus we have completed the proof of the lemma.
\end{proof}

We shall also need the following ``Weyl differencing" lemma.

\begin{lemma} \label{weyldifflem}
For any complex numbers $z_n$, we have
\[ \left| \sum_{a <n < b} z_n \right|^2 \leq \left( 1 + \frac{b-a}{Q} \right) \sum_{|q|<Q} \left( 1- \frac{|q|}{Q} \right) \sum_{a < n, n+q<b} z_{n+q} \overline{z_{n}}, \]
where $Q$ is any positive integer.
\end{lemma}

\begin{proof}
This is Lemma 8.17 in \cite{HIEK}.
\end{proof}

\section{Exponential sums with Hecke eigenvalues}
In this section, we consider exponential sums of the form
\begin{equation} \label{Sdef}
S=\sum\limits_{N<n\le N'} \lambda(n) e(f(n)),
\end{equation}
where $3\le N<N'\le 2N$ and $f\in C^{\infty}([N/2,3N])$ satisfies
\begin{equation} \label{fconds}
\left|f^{(k)}(x)\right|\ll_k \frac{T}{N^{k}} \quad \mbox{for all } x\in [N/2,3N] \mbox{ and } k\in \mathbb{N}_0  
\end{equation}
with some 
\begin{equation} \label{Tcond}
T\ge N^{3/4}.
\end{equation}

We shall prove the following lemma.

\begin{lemma} \label{Slemma}
With $S$ defined in \eqref{Sdef} and the conditions \eqref{fconds} and \eqref{Tcond} satisfied, we have
\begin{equation} \label{Sesti}
S \ll N^{2/3+\varepsilon}T^{5/18}+N^{5/6}T^{-5/18},
\end{equation}
where the implied constant depends on $\varepsilon$ only.
\end{lemma}

\begin{proof}
We first do a ``Weyl differencing", where we introduce an extra smooth weight function 
$\Phi\in C^{\infty}(\mathbb{R})$, compactly supported in $[N/2,5N/2]$ and satisfying
$$
\Phi^{(k)}(x)\ll_k N^{-k} \quad \mbox{for all } x\in \mathbb{R}^+ \mbox{ and } k\in \mathbb{N}_0
$$
and
$$
\Phi(x)=1 \quad \mbox{for } N\le x\le N'.
$$
Let $Q$ be any positive integer and set
$$
z_n:=\begin{cases}
      1 & \mbox{ if } N<n\le N' , \\ 0 & \mbox{ otherwise.}
     \end{cases}
$$
Then we have
$$
S=\sum\limits_{n} z_n\lambda(n) e(f(n))=\sum\limits_{n} z_{n+q}\lambda(n+q) e(f(n+q))
$$
for any $q\in \mathbb{Z}$. We sum this up over $q$ with $0\le q<Q\le N/2$, getting
$$
QS=\sum\limits_{N-Q<n<N'}\ \sum\limits_{0\le q<Q} z_{n+q}\lambda(n+q) e(f(n+q)).
$$
Hence, by Cauchy's inequality,
$$
Q^2|S|^2\le (N'-N+Q)\sum\limits_n \left| \sum\limits_{0\le q<Q} z_{n+q}\lambda(n+q) e(f(n+q)) \right|^2.
$$
It follows that 
$$
Q^2|S|^2\ll N\sum\limits_{N<n\le N'-Q} \left| \sum\limits_{0\le q<Q} \lambda(n+q) e(f(n+q)) \right|^2 + Q^3N^{1+\varepsilon}
$$
and further
$$
Q^2|S|^2\ll N\sum\limits_{n} \Phi(n)^2 
\left| \sum\limits_{0\le q<Q} \lambda(n+q) e(f(n+q)) \right|^2 + Q^3N^{1+\varepsilon}.
$$

Expanding the square on the right-hand side and setting
$$
G_{q_1,q_2}(m_1,m_2):=\Phi(m_1-q_1)\Phi(m_2-q_2) \quad \mbox{and} \quad F_{q_1,q_2}(m):=f(m)-f(m+q_1-q_2)
$$
gives
\begin{equation} \label{afterweyl}
\begin{split}
Q^2|S|^2\ll  N\sum\limits_{0\le q_1<Q}\ & \sum\limits_{0\le q_2<Q} 
\sum\limits_{\substack{m_1,m_2\\ m_1-m_2=q_1-q_2}} G_{q_1,q_2}(m_1,m_2)
\lambda(m_1) \lambda(m_2) e\left(F_{q_1,q_2}(m_1)\right) + \\
& +Q^3N^{1+\varepsilon}.
\end{split}
\end{equation}
Now we impose the condition that
$$
Q\ge \frac{N}{T}.
$$
Then a simple computation shows that
$$
\frac{\dif^{i+j}}{\dif x^i \dif y^j} G_{q_1,q_2}(x,y)
e\left(F_{q_1,q_2}(x)\right)\ll_{i,j} \left(\frac{TQ}{N^2}\right)^i N^{-j}
$$
for
$$ 
0\le q_1,q_2<Q, \quad
\frac{N}{2}+q_1\le x\le \frac{5N}{2}+q_1, \quad \frac{N}{2}+q_2\le y\le \frac{5N}{2}+q_2.
$$

Now if $q_1\not=q_2$, we use Lemma \ref{shiftedconvolution} with
$$
g(x,y):=G_{q_1,q_2}(x,y) e\left(F_{q_1,q_2}(x)\right)
$$
and 
$$
a=b=1, \quad X:=N,\quad Y:=N,\quad P:=\frac{TQ}{N}
$$
to deduce that the inner double sum on the right-hand side of \eqref{afterweyl} is
\begin{equation}
\sum\limits_{\substack{m_1,m_2\\ m_1-m_2=q_1-q_2}} G_{q_1,q_2}(m_1,m_2)
\lambda(m_1) \lambda(m_2) e\left(F_{q_1,q_2}(m_1)\right)\ll (TQ)^{5/4}N^{-1/2+\varepsilon}.
\end{equation}
If $q_1=q_2$, then we have the trivial bound
\begin{equation}
\sum\limits_{\substack{m_1,m_2\\ m_1-m_2=q_1-q_2}} G_{q_1,q_2}(m_1,m_2)
\lambda(m_1) \lambda(m_2) e\left(F_{q_1,q_2}(m_1)\right)=\sum\limits_{m} \Phi(m-q_1)^2
\lambda(m_1)^2 \ll N^{1+\varepsilon}.
\end{equation}

Combining everything in this section, we obtain
\begin{equation} \label{Sest}
S\ll (TQ)^{5/8}N^{1/4+\varepsilon}+\frac{N}{Q^{1/2}}+N^{1/2}Q^{1/2}
\end{equation}
under the condition
\begin{equation} \label{Qcond}
\frac{N}{T}\le Q\le \frac{N}{2}.
\end{equation}
Now we choose
$$
Q:=\frac{N^{2/3}}{T^{5/9}}.
$$
Then, by $N\ge 3$ and \eqref{Tcond}, the condition in \eqref{Qcond} is satisfied, and we get \eqref{Sesti}.
\end{proof}

\section{Reduction to exponential sums}

Using $\lambda(n)\ll n^{\varepsilon}$, partial summation, and the fact that every cusp form can be written as a linear combination of finitely many Hecke eigenforms, Theorem \ref{mainresult}, our main result, can be easily deduced from the following result whose proof will be the object of the remainder of this paper. \newline

\begin{theorem} \label{mainresultmodified}
Let $g : [1,\infty)\rightarrow [1,\infty)$ be a function satisfying the conditions (i) -- (iii) in Section 1.  Suppose that the inverse function of $g$, $f$, satisfies the condtions (a) -- (e) in Section 1.  Let $\lambda(n)$ be the normalized $n$-th Fourier coefficient of a Hecke eigenform for the full modular group.  By $\Lambda(n)$, we denote the von Mangoldt function. Then there exists a positive constant $C$ depending on the cusp form such that
\begin{equation} \label{goal}
\sum\limits_{n\le N} \Lambda\left(\left[g(n)\right]\right)\lambda\left(\left[g(n)\right]\right) \ll
N\exp(-C\sqrt{\log N}),
\end{equation}
where the implied $\ll$-constant depends only on $C$ and the cusp form.
\end{theorem}

In this section, we reduce the left-hand side of \eqref{goal} to  exponential sums.  We recall that $f:=g^{-1}$ denotes the 
function inverse to $g$.
Let $m,n\in \mathbb{N}$. Then $\left[g(n)\right]=m$ is equivalent to
$$
-f(m+1)< -n\le -f(m).
$$
Therefore, we have
\begin{equation} \label{rewritten}
\sum\limits_{n\le N} \Lambda\left(\left[g(n)\right]\right)\lambda\left(\left[g(n)\right]\right) =\sum\limits_{g(1)\le m\le g(N)} 
\left(\left[-f(m)\right]-\left[-f(m+1)
\right]\right)\Lambda(m)\lambda(m)+O(\log N).
\end{equation}
Breaking into dyadic intervals and using that $g$ is monotonically increasing, it hence suffices to prove that
\begin{equation} \label{S}
S:=\sum\limits_{n\sim g(N)} \left(\left[-f(n)\right]-\left[-f(n+1)
\right]\right)\Lambda(n)\lambda(n)\ll N\exp(-C\sqrt{\log N})
\end{equation}
for any $N>1$. We write the above sum $S$ in the form
\begin{equation} \label{S12}
S=S_1+S_2,
\end{equation}
where
$$
S_1=\sum\limits_{n\sim g(N)} \left(f(n+1)-f(n)\right)\Lambda(n)\lambda(n)
$$
and
$$
S_2=\sum\limits_{n\sim g(N)} \left(\psi\left(-f(n+1)\right)-\psi\left(-f(n)\right)\right)\Lambda(n)\lambda(n),
$$
with $\psi(n)$ being the saw-tooth function in Lemma~\ref{Valer}. \newline

By \eqref{fder} and the mean value theorem, we have the bounds
$$
f(x+1)-f(x)\ll \frac{f(x)}{x} \quad \mbox{and} \quad
\frac{\dif}{\dif x} \left(f(x+1)-f(x)\right) \ll\frac{f(x)}{x^2}
$$
for all $x$ in the image of $g$. Hence,
using partial summation, $f\circ g(x)=x$ and $g(x)\ll x^{30/29-\varepsilon}$, we deduce from Lemma \ref{cuspatprimes} that
$$
S_1\ll N\exp(-C\sqrt{\log N}),
$$
where the implied constant depends only on $C$ and the cusp form. \newline

Our treatment of the sum $S_2$ begins like in \cite{GK}. By Lemma \ref{Valer}, we have the following. For any $J>0$ there exist functions $\psi^*$ and $\delta$, with $\delta$ non-negative, such that
$$
\psi(x)=\psi^*(x)+O( \delta(x) ),
$$
where
$$
\psi^*(x)=\sum\limits_{1\le |j|\le J} a(j)e(jx), \quad
\delta(x)=\sum\limits_{|j|\le J} b(j)e(jx)
$$
with
$$
a(j)\ll j^{-1}, \quad b(j)\ll J^{-1}.
$$
Consequently,
\begin{eqnarray*}
S_2&=&\sum\limits_{n\sim g(N)} \left(\psi^*\left(-f(n+1)\right)-\psi^*\left(-f(n)\right)\right)\Lambda(n)\lambda(n)+
O\left((\log N) \sum\limits_{n\sim g(N)} \left(\delta\left(-f(n+1)\right)+\delta\left(-f(n)\right)\right)  \right)\\
&=& S_3+O(S_4),
\end{eqnarray*}
say. We fix a small $\eta>0$ and set
\begin{equation} \label{J}
J:=\frac{g(N)}{N}\cdot N^{\eta}.
\end{equation}
Then, using \eqref{fder}, Lemma \ref{deltaest} and $g(N)\ll N^{30/29-\varepsilon}$, we obtain
$$
S_4\ll N^{1-\eta/2}.
$$

The remaining task is to prove that
$$
S_3\ll N^{1-\eta/2},
$$
provided that $\eta$ is sufficiently small. We write
$$
S_3=\sum\limits_{1\le |j|\le J} \sum\limits_{n\sim g(N)}  \Lambda(n)\lambda(n)a(j)\phi_j(n)e(-jf(n)),
$$
where $\phi_j(x)=1-e(j(f(x)-f(x+1)))$. Using partial summation and the bounds $a(j)\ll j^{-1}$ and
$$
\phi_j(x)\ll \frac{jf(x)}{x} \quad \mbox{and} \quad \frac{\dif}{\dif x} \phi_j(x)\ll \frac{jf(x)}{x^2},
$$
we deduce that it suffices to prove that
$$
\sum\limits_{1\le |j|\le J} \left|\sum\limits_{n\sim g(N)}  \Lambda(n)\lambda(n)e(-jf(n))\right| \ll g(N)N^{-\eta/2}.
$$
Replacing $g(N)$ by $N$ and $N$ by $f(N)$, taking the definition of $J$ in \eqref{J} into account, dividing the summation interval 
$1\le |j|\le J$ into $O(\log 2J)$ dyadic intervals, and using the facts that $e(-x)=\overline{e(x)}$ and the Hecke eigenvalues are real, 
we see that the above bound holds if
\begin{equation} \label{goal1}
\sum\limits_{h\sim H} \left| \sum\limits_{n\sim N} \Lambda(n)\lambda(n)e\left(hf(n)\right)\right| \ll N^{1-\eta}
\end{equation}
for any $N\ge 1$ and $1\le H\le N^{1+\eta}f(N)^{-1}$. The following lemma reduces the term on the left-hand side of \eqref{goal1} to trilinear exponential sums.

\begin{lemma} \label{bilinearsums}
Suppose that $u$, $v$ and $z$ are real parameters satisfying the conditions
\begin{equation} \label{uvzcond}
3\le u<v<z<2N,\ z-1/2\in \mathbb{N},\ z\ge 4u^2,\ N\ge 32z^2u,\ v^3\ge 64N. \end{equation}
Suppose further that $1\le Y\le N$, $XY=N$ and $H\ge 1$. Assume that $A_m$, $B_n$ and $C_h$ are complex numbers.  For $d\in \mathbb{N}$ set
\begin{equation} \label{Kddef}
K_d:=\mathop{\sum\limits_{m\sim X/d}\ \sum_{ n\sim Y/d}}_{mn\sim N/d^2}\ \sum\limits_{h\sim H} \
A_m C_h \lambda(n)e\left(hf\left(d^2mn\right)\right)
\end{equation}
and
\begin{equation} \label{Lddef}
L_d:=\mathop{\sum\limits_{m\sim X/d}\ \sum_{n\sim Y/d}}_{mn\sim N/d^2}\ \sum\limits_{h\sim H}\  A_{m} B_n C_h 
e\left(hf\left(d^2mn\right)\right).
\end{equation}
Then the estimate \eqref{goal1} holds if we uniformly have
\begin{equation} \label{Kdbound}
K_d\ll N^{1-3\eta}d^{-1}\ \  \mbox{ for } Y\ge z,\ d\le 2Y \; \mbox{and any complex} \; A_m,C_h\ll 1
\end{equation}
and
\begin{equation} \label{Ldbound}
L_d\ll N^{1-3\eta}d^{-1}\ \  \mbox{ for } u\le Y\le v,\ d\le 2Y \; \mbox{and any complex} \; A_m,B_n,C_h\ll 1.
\end{equation}
\end{lemma}

\begin{proof} We first write
$$
\sum\limits_{h\sim H} \left| \sum\limits_{n\sim N} \Lambda(n)\lambda(n)e\left(hf(n)\right)\right| =
\sum\limits_{h\sim H} c_h \sum\limits_{n\sim N} \Lambda(n)\lambda(n)e\left(hf(n)\right),
$$
where $c_h$ are suitable complex numbers with $|c_h|=1$. We further set
$$
r(n):=\lambda(n)\sum\limits_{h\sim H} c_he\left(hf(n)\right)
$$
so that
$$
\sum\limits_{h\sim H} \left| \sum\limits_{n\sim N}
\Lambda(n)\lambda(n)e\left(hf(n)\right)\right|=\sum\limits_{n\sim N}
\Lambda(n)r(n).
$$
Now, by Lemma \ref{Heath}, the bound \eqref{goal1} holds if
\begin{equation} \label{KL}
K\ll N^{1-2\eta}\quad \mbox{and} \quad L\ll N^{1-2\eta}
\end{equation}
under the conditions of the same lemma. Here $K$ and $L$ are defined as in \eqref{Kdef} and \eqref{Ldef}. We may rewrite these terms in the form
$$
K=\mathop{\sum\limits_{m\sim X} \sum_{n\sim Y}}_{mn\sim N} \sum\limits_{h\sim H} a_mc_h\lambda(mn)e\left(hf(mn)\right)
$$
and
$$
L=\mathop{\sum\limits_{m\sim X} \sum_{n\sim Y}}_{mn\sim N} \sum\limits_{h\sim H} a_m b_nc_h\lambda(mn)e\left(hf(mn)\right).
$$
Using the multiplicative property of Hecke eigenvalues, Lemma \ref{multhecke}, we have
\begin{equation} \label{rewriteK}
K=\sum\limits_{d\le 2Y} \mu(d) \mathop{\sum\limits_{m\sim X/d}\ \sum_{n\sim Y/d}}_{mn\sim N/d^2}\ \sum\limits_{h\sim H} \
a_{dm} \lambda(m) c_h \lambda(n) e\left(hf(d^2mn)\right)
\end{equation}
and
\begin{equation} \label{rewriteL}
L=\sum\limits_{d\le 2Y} \mu(d) \mathop{\sum\limits_{m\sim X/d}\ \sum_{ n\sim Y/d}}_{mn\sim N/d^2}\ \sum\limits_{h\sim H} \
a_{dm} \lambda(m) b_{dn} \lambda(n) c_h e\left(hf(d^2mn)\right).
\end{equation}
Now, \eqref{KL} follows from \eqref{Kdbound}, \eqref{Ldbound},\eqref{rewriteK}, \eqref{rewriteL} and the bound
$\lambda(n)\ll n^{\varepsilon}$.
 \end{proof}

In the following sections, we shall estimate the terms $K_d$ and $L_d$.

\section{Estimation of $L_d$}

Our task in this section is to estimate $L_d$, defined in \eqref{Lddef}.

\begin{lemma} \label{Ldcondlemma}
For every sufficiently small and fixed $\eta >0$, we have
\begin{equation} \label{Lddesiredest}
 L_d \ll N^{1-3 \eta} d^{-1}
 \end{equation}
provided that $f(N) \geq N^{8/9+30 \eta}$, $1 \leq H \leq N^{1+\eta}f(N)^{-1}$, $1 \leq d \leq 2Y$ and
\begin{equation} \label{Ycond1}
\frac{N^{2+100 \eta}}{f(N)^2} \le Y \leq \frac{f(N)^6}{N^{5+100 \eta}}.
\end{equation}
\end{lemma}

\begin{proof}
From \eqref{Lddef}, we have
\[ L_d = \mathop{\sum_{m\sim X/d}\ \sum_{n\sim Y/d}}_{mn \sim N/d^2}\ \sum_{h\sim H} A_m B_n C_h e (h f(d^2 mn) ), \]
with
\[ A_m \ll m^{\varepsilon}, \; B_n \ll n^{\varepsilon} \; \mbox{and} \; C_h \ll h^{\varepsilon}. \]
Using Cauchy's inequality, we get
\begin{equation} \label{cauchy}
L_d^2 \ll N^{\varepsilon} \frac{X}{d} H \sum_{h} \sum_m \left| \sum_n B_n e \left( h f(d^2mn) \right) \right|^2.
\end{equation}
Using the ``Weyl differencing", Lemma~\ref{weyldifflem}, we have
\begin{equation} \label{weyldiff}
\begin{split}
&  \left| \sum_n B_n e \left( h f(d^2mn) \right) \right|^2 \\
& \leq \left( 1 + \frac{Y/d}{Q} \right) \sum_{|q| <Q} \left( 1 - \frac{|q|}{Q} \right) \sum_{\substack{n \sim Y/d \\ n+q \sim Y/d}} B_{n+q} \overline{B_n} e \left( h \left( f(d^2m(n+q)) - f(d^2mn) \right) \right) ,
\end{split}
\end{equation}
where $Q$ is a parameter to be chosen later and satisfies the condition
\begin{equation} \label{Qbound}
Q \leq Y/d.
\end{equation}
Inserting the above into \eqref{cauchy}, we have, since $XY=N$ and $Q < Y/d$,
\begin{equation} \label{afterweyl}
 L_d^2 \ll N^{\varepsilon} \left( \frac{H^2}{Q} \frac{N^2}{d^4}  + \frac{H}{Q} \frac{N}{d^2} \sum_h \sum_{0 < |q| < Q} \sum_{\substack{n \sim Y/d \\ n+q \sim Y/d}} \left| B_{n+q} \overline{B_n} \sum_{m \in I} e \left( h \left( f(d^2m(n+q)) - f(d^2mn\right) \right) \right| \right).
 \end{equation}
The first term on the right-hand side of \eqref{afterweyl} is the contribution from $q=0$ and $I$ denotes the interval defined by the conditions
\[ m \sim X/d, \; mn \sim N/d^2 \; \mbox{and} \; m(n+q) \sim N/d^2. \]
Note that
\begin{equation*}
\begin{split}
 \frac{\dif^2}{\dif m^2} \left( f(d^2 m (n+q))  - f(d^2m n) \right) & = d^4(n+q)^2 f''(d^2m(n+q)) - d^4 n^2 f'' (d^2mn) \\ & =  d^4 q n_0 \left( 2 f''(d^2n_0m) + n_0 d^2m f'''(d^2n_0m) \right),
 \end{split}
 \end{equation*}
by the mean-value theorem applied to the function $\tilde{f}(x)=x^2f''(d^2xm)$, for some $n_0$ between $n$ and $n+q$.  Using \eqref{f2ndder}, it follows that
\begin{equation*}
 h \frac{\dif^2}{\dif m^2} \left( f(d^2 m(n+q) ) - f (d^2  m n) \right) \asymp hd^3 |q| Y \frac{f(N)}{N^2}.
 \end{equation*}
Hence, Lemma~\ref{vandercorput} gives
\begin{equation*}
\begin{split}
 \sum_m  e & \left( h ( f(d^2 n_1 m) - f (d^2 n_2 m) ) \right) \\
 & \ll X h^{1/2} d^{1/2} |q|^{1/2} Y^{1/2} \frac{f(N)^{1/2}}{N} + \frac{N}{h^{1/2}d^{3/2} |q|^{1/2}Y^{1/2} f(N)^{1/2}} .
 \end{split}
 \end{equation*}
Now inserting the above estimate into \eqref{afterweyl}, summing over all the relevant variables and mindful of $XY=N$, we get that
\begin{equation} \label{beforeQchoice}
 L_d^2 \ll N^{2\varepsilon} \left( \frac{H^2N^2}{Qd^4} + H^{5/2} N Q^{1/2} Y^{1/2} f(N)^{1/2} d^{-5/2} + \frac{H^{3/2} N^2 Y^{1/2}}{Q^{1/2} f(N)^{1/2}d^{9/2}} \right).
 \end{equation}
To equalize the first two terms above, we set
\begin{equation} \label{Qchoice1} 
Q = H^{-1/3} N^{2/3} f(N)^{-1/3} Y^{-1/3}.
\end{equation}
If the lower bound for $Y$ in \eqref{Ycond1} holds, this choice of $Q$ is in accordance with  \eqref{Qbound}.  It follows from \eqref{beforeQchoice} that
\begin{eqnarray*} 
L_d^2 & \ll & \frac{N^{2\varepsilon}}{d^2} \left( H^{7/3} N^{4/3} f(N)^{1/3} Y^{1/3} + H^{5/3} N^{5/3} Y^{2/3} f(N)^{-1/3} \right) \\
& \ll & \frac{N^{2\varepsilon}}{d^2} \left( N^{11/3+7/3 \eta} f(N)^{-2} Y^{1/3} + N^{10/3+5/3\eta} f(N)^{-2} Y^{2/3} \right),
\end{eqnarray*}
since $H \leq N^{1+\eta}f(N)^{-1}$.  Taking the square-root, we have the desired estimate in \eqref{Lddesiredest} provided that
\begin{equation} \label{Yrange1}
\frac{N^{1/2+100 \eta}}{f(N)^{1/4}} \leq Y \leq \frac{f(N)^6}{N^{5+100 \eta}},
\end{equation}
where we use that $f(N)\le N$.

If, instead of choosing $Q$ as in \eqref{Qchoice1}, we simply set
\begin{equation} \label{Qchoice2}
Q = Y/d
\end{equation}
which certainly satisfies the requirement in \eqref{Qbound}, then from \eqref{beforeQchoice}, repeating the above computations with this choice of $Q$, we arrive at the estimate
\begin{eqnarray*}
 L_d^2 & \ll & \frac{N^{2\varepsilon}}{d^2} \left( N^2 Y^{-1} H^2 + NY H^{5/2} f(N)^{1/2} + N^2 H^{3/2} f(N)^{-1/2} \right) \\
 & \ll & \frac{N^{2\varepsilon}}{d^2} \left( N^{4+2\eta} f(N)^{-2} Y^{-1} +  N^{7/2+5/2\eta} Y f(N)^{-2} + N^{7/2 + 3/2\eta} f(N)^{-2} \right).
 \end{eqnarray*}
This gives the desired majorant in \eqref{Lddesiredest} if
\begin{equation} \label{Yrange2}
 \frac{N^{2+100 \eta}}{f(N)^2} \leq Y \leq \frac{f(N)^2}{N^{3/2+100 \eta}}.
 \end{equation}

We note that
\[ f(N) \geq N^{8/9+50 \eta} \]
implies
\[  \frac{N^{2+100 \eta}}{f(N)^2} \le \frac{N^{1/2+100 \eta}}{f(N)^{1/4}} \leq \frac{f(N)^2}{N^{3/2+100 \eta}} \leq \frac{f(N)^6}{N^{5+100 \eta}}.\]
Now joining the two $Y$-ranges in \eqref{Yrange1} and \eqref{Yrange2}, we get the lemma.
\end{proof}

\section{Estimation of $K_d$}
For small $Y$, we cannot directly exploit the smooth exponential sum over $n$ with Hecke eigenvalue $\lambda(n)$. In this case, we treat $\lambda(n)$ like an arbitrary coefficient
and hence $K_d$ like $L_d$, obtaining the following result.
\begin{lemma} \label{Kdcondlemma0}
For every sufficiently small fixed $\eta>0$, we have
\begin{equation} \label{Kddesiredest}
 K_d\ll N^{1-3\eta}d^{-1},
 \end{equation}
provided that $f(N)\ge N^{8/9+30\eta}$, $1\le H\le N^{1+\eta}f(N)^{-1}$, $1\le d\le 2Y$ and
\begin{equation} \label{KYcond1}
 \frac{N^{6+100 \eta}}{f(N)^{6}} \le Y\le \frac{f(N)^{2}}{N^{1+100\eta}}.
\end{equation}
\end{lemma}

\begin{proof}
This can be proved in essentially the same way as Lemma \ref{Ldcondlemma}, but with the roles of $X$ and $Y$ reversed. Similarly as in Lemma \ref{Ldcondlemma}, we get that $K_d\ll N^{1-3\eta}d^{-1}$, provided that
\[ \frac{N^{2+100 \eta}}{f(N)^2} \le X \leq \frac{f(N)^6}{N^{5+100 \eta}}. \]
These inequalities are equivalent to \eqref{KYcond1} since $XY=N$.
\end{proof}

For large $Y$, we employ Lemma \ref{Slemma} to deduce the following.

\begin{lemma} \label{Kdcondlemma1}
For every sufficiently small fixed $\eta>0$, we have
\[ K_d\ll N^{1-3\eta}d^{-1}, \]
provided that $f(N)\ge N^{3/4+10\eta}$, $1\le H\le N^{1+\eta}f(N)^{-1}$, $1\le d\le 2Y$ and
\begin{equation} \label{KYcond2}
Y\ge N^{23/6+100\eta}f(N)^{-3}.
\end{equation}
\end{lemma}

\begin{proof}
We note that for every $k\in \mathbb{N}$, we have
\begin{equation}
\frac{\dif^k}{\dif y^k} hf\left(d^2my\right)  \asymp \frac{hf(N)}{y^k}
\end{equation}
by \eqref{fder}. Thus, we may apply Lemma \ref{Slemma} with $N$ replaced by $Y/d$ and $T=hf(N)$
to the sum over $n$, provided that $f(N)\ge N^{3/4}$. Summing up the resulting estimate trivially over $h$ and $m$, we obtain
\begin{equation}
K_d\ll H\cdot \frac{X}{d} \cdot \left(\frac{Y^{2/3+\varepsilon}}{d^{2/3}}\cdot (Hf(N))^{5/18}+\frac{Y^{5/6}}{d^{5/6}}\cdot (Hf(N))^{-5/18}\right).
\end{equation}
Therefore, the lemma follows upon noting that $H\le N^{1+\eta}f(N)^{-1}$, $XY=N$ and $f(N)\ge N^{3/4+10\eta}$.
\end{proof}

Combining the above Lemmas \ref{Kdcondlemma0} and \ref{Kdcondlemma1}, we arrive at the following conclusion.

\begin{lemma} \label{finalKdcondlemma}
For every sufficiently small fixed $\eta>0$, we have
\[ K_d\ll N^{1-3\eta}d^{-1}, \]
provided that $f(N)\ge N^{29/30+100\eta}$, $1\le H\le N^{1-\gamma+\eta}$, $1\le d\le 2Y$ and
$$
\frac{N^{6+100 \eta}}{f(N)^{6}}\le Y\le 2N.
$$
\end{lemma}
\begin{proof} Clearly, the $Y$-ranges in Lemma \ref{Kdcondlemma0} and \ref{Kdcondlemma1} overlap if  $f(N)\ge N^{29/30+100\eta}$. This proves Lemma 
\ref{finalKdcondlemma}.
\end{proof}

We point out that the condition \eqref{fcond} on $f$ arises from Lemma \ref{finalKdcondlemma}.

\section{Proof of the main result \label{finalproof}}

\begin{proof}[Proof of Theorems~\ref{mainresultmodified} and~\ref{mainresult}] We recall that Theorem~\ref{mainresultmodified} and hence Theorem~\ref{mainresult}, our main result, holds if \eqref{goal1} is valid for any $N\ge 1$ and $1\le H\le N^{1+\eta}f(N)^{-1}$. Here $f$ satisfies the conditions (a) - (e) in the introduction, and $\eta$ is sufficiently small, which we assume in the following. Furthermore, in Lemma \ref{bilinearsums} we formulated some conditions on bilinear sums $K_d$ and $L_d$ under which \eqref{goal1} holds. In the following, we check that these conditions are satisfied. \newline

We choose the parameters $u$, $v$ and $z$ in Lemma \ref{bilinearsums} as follows.
\begin{eqnarray*}
u&:=&N^{2+100\eta}f(N)^{-2},\\
v&:=&4N^{1/3},\\
z&:=&\left[f(N)N^{-1/2-100\eta}\right]+1/2.
\end{eqnarray*}
The parameters $u$, $v$ and $z$, so chosen, indeed satisfy the conditions in \eqref{uvzcond} if $f(N)\ge N^{9/10+\varepsilon}$ and $\eta$ is sufficiently small. Moreover, the conditions \eqref{Kdbound} and \eqref{Ldbound} hold by Lemmas \ref{Ldcondlemma} and \ref{finalKdcondlemma} since 
$$
4N^{1/3} \leq \frac{f(N)^6}{N^{5+100 \eta}} \quad \mbox{and} \quad \frac{N^{6+100 \eta}}{f(N)^{6}}\le \frac{f(N)}{N^{1/2+100\eta}}
$$ 
if $f(N)\ge N^{13/14+\varepsilon}$ and $\eta$ is sufficiently small. This completes the proof.
\end{proof}

\noindent{\bf Acknowledgments.} S. B. thanks the restaurant ``Kreuzgang'' on the Marktplatz in G\"ottingen for enabling him to eat his beloved liver everyday.  During this work, S. B. was supported by an ELCC Grant 25871 and L. Z. by an AcRF Tier 1 Grant at Nanyang Technological University.

\bibliography{biblio}
\bibliographystyle{amsxport}

\vspace*{.5cm}

\noindent\begin{tabular}{p{8cm}p{8cm}}
Stephan Baier & Liangyi Zhao \\
Mathematisches Institut & Division of Mathematical Sciences \\
Universit\"at G\"ottingen & School of Physical and Mathematical Sciences \\
Bunsenstr.\ 3--5, & Nanyang Technolgoical University\\
G\"ottingen 37073 Geramny & Singapore 637371\\
Email: {\tt sbaier@uni-math.gwdg.de} & Email: {\tt lzhao@pmail.ntu.edu.sg} \\
\end{tabular}

\end{document}